\documentclass[12pt]{amsart}

\usepackage{fullpage} 
\usepackage{parskip} 
\usepackage{tikz-cd} 
\usepackage{amsmath, amssymb, amsthm}
\usepackage{esint}
\usepackage{hyperref}
\hypersetup{
	colorlinks = false,
	linkbordercolor = {white},
	citebordercolor = {white},
}
\usepackage[utf8]{inputenc}
\usepackage[english]{babel}
\usepackage[alphabetic]{amsrefs}
\usepackage{geometry}
\makeatother

\theoremstyle{definition}

\def\fnum{equation} 
\newtheorem{Thm}[\fnum]{Theorem}
\newtheorem{Cor}[\fnum]{Corollary}

\newtheorem{Lem}[\fnum]{Lemma}

\newtheorem{Def}[\fnum]{Definition}

\newtheorem{Pro}[\fnum]{Proposition}

\numberwithin{equation}{section}

\newcommand{\Ker}{{\text{Ker}}}

 \newcommand{\M}{\ensuremath{\mathcal{M}}}

 \newcommand{\ba}{\begin{align*}}
 \newcommand{\ea}{\end{align*}}
 \newcommand{\na}{\nabla}
\newcommand{\la}{\langle}
\newcommand{\ra}{\rangle}

\newcommand{\ovl}[1]{\overline{#1}}
\newcommand{\udl}[1]{\underline{#1}}
\newcommand{\pair}[1]{\left\langle #1\right\rangle}
\newcommand{\bb}[1]{\mathbb{#1}} \newcommand{\td}[1]{\widetilde{#1}}
\newcommand{\abs}[1]{\left|#1\right|}
\newcommand{\pr}[1]{\left(#1\right)}

\def\RR{{\bold R}}

\newcommand{\cC}{{\mathcal{C}}}

\newcommand{\cL}{{\mathcal{L}}}

\newcommand{\cS}{{\mathcal{S}}}

\newcommand{\D}{\Delta}

\newcommand{\n}{\nabla}

\newcommand{\dl}{\mathcal L}
\newcommand{\sbst}{\subseteq}
\newcommand{\h}{\textbf H}

\newcommand{\R}{\textbf{R}}
\newcommand{\bd}{\partial}

\newcommand{\N}{\textbf n}

\newcommand{\supp}{\text{supp}}

\newcommand{\gr}{\text{graph}~}
\newcommand{\rl}{\widetilde{\textbf{r}}_l}
\newcommand{\rr}{\textbf r}
\newcommand{\np}{\n^\perp}
\newcommand{\dlp}{\mathcal L^\perp}
\newcommand{\id}{\text{id}}

 \newcommand{\CS}[1]{\mathcal{CS}_{-1}^{#1,\alpha}(\Sigma)}
 
 \newcommand{\sN}{\mathcal N}

\title[Uniqueness of conical MCF]
{Uniqueness of conical singularities\\for mean curvature flows}

\author{Tang-Kai Lee}%
\author{Xinrui Zhao}%
\address{MIT, Dept. of Math.\\
77 Massachusetts Avenue, Cambridge, MA 02139-4307.}

\date{\today}

\email{tangkai@mit.edu and xrzhao@mit.edu}
\begin{document}

\maketitle

\begin{abstract} 
In this paper, we prove the uniqueness of asymptotically conical tangent flows in all codimensions. 
This is based on an early work of Chodosh-Schulze, who proved the uniqueness in the hypersurface case.
\end{abstract}

\section{\bf Introduction}
Mean curvature flow (MCF) is a family of immersions $\mathbf{x}\colon M^n\times I\to \bb R^{n+k}$ such that 
$\bd_t \mathbf{x} = \mathbf{H}$
where $\mathbf{H}$ is the mean curvature vector of the immersed submanifold $M_t=\mathbf{x}(M,t).$
To understand the structure of a MCF, studying its singularities is the most important problem since they are unavoidable in general.
For example, every MCF starting from a closed hypersurface develops a finite-time singularity.
The singularities are usually investigated via a blowup technique.
This leaves the most fundamental question about the \textit{uniqueness} of these singularity models.

To elaborate, given a MCF $M_t$ with a spacetime singularity at $(\textbf 0,0)\in\bb R^{n+k}\times\bb R,$ we can define a sequence of MCFs $M^i_t:= \lambda_i M_{\lambda_i^{-2}t}$ by rescaling using any sequence $\lambda_i\to\infty.$
Based on Huisken's monotonicity \cite{H90} and Brakke's compactness \cite{B78}, White \cite{W94} and Ilmanen \cite{I95} proved that $M_t^i$ weakly converges to a self-similar limiting flow, called a \textit{tangent flow}, after replacing $\lambda_i$'s with a subsequence.
When the limiting flow is smooth, it is generated by a \textit{self-shrinker}, a submanifold satisfying $\h=-\frac{x^\perp}{2}.$
In general, this limiting flow depends on the sequence we choose to rescale the flow.
That is to say, choosing different subsequences may result in different tangent flows at the same singularity.
Determining the uniqueness of tangent flows is a crucial question if we would like to understand the singularity structure of the flow. 
See e.g., \cite{S14, CM15, CM19, Z20, CS, LSS}.
A similar problem also plays an important role in the study of minimal submanifolds, that is, the uniqueness of tangent cones. 
For related references, one can see \cite{AA, S83, H97, S08, CM14}.

In this paper, we will focus on asymptotically conical self-shrinkers.
An \textit{asymptotically conical} self-shrinker is a self-shrinker that is (smoothly) asymptotic to a cone at infinity. 
Such kinds of self-shrinkers are important in understanding two-dimensional self-shrinkers, see for example \cite{W14}. 
Asymptotically conical self-shrinkers have other interesting properties. 
For example, given any cone in $\bb R^{n+k}$, there is at most one asymptotically conical self-shrinker asymptotic to it even in the case that the codimension is higher than one \cite{W14,Kh}. 
Also, asymptotically conical self-shrinkers as tangent flows of hypersurface MCFs are unique, proven by Chodosh-Schulze \cite{CS}. 
Following their framework and combining with an extension theorem that originated from White's curvature estimate, we can prove a higher-codimensional analogue of Chodosh-Schulze's result.

\begin{Thm}[Uniqueness of asymptotically conical tangent flows]\label{prop:uniqueness}
Let $\Sigma^n\sbst \bb R^{n+k}$ be a multiplicity one asymptotically conical self-shrinker and $(M_t^n)_{t\in(-T,0)}$ a MCF in $\bb R^{n+k}$ such that a tangent flow at $(\textbf 0,0)$ is the MCF generated by $\Sigma.$ 
Then it is the unique tangent flow of $(M_t^n)_{t\in(-T,0)}$ at $(\textbf 0,0)$.
\end{Thm}

Uniqueness properties of tangent flows for geometric flows  usually follow from a \L ojasiewicz-type inequality.
For geometric applications, these kinds of inequalities were first proven by Simon \cite{S83} by using the Lyapunov-Schmidt reduction and applying the original inequality invented by \L ojasiewicz \cite{L} for real analytic functions.
As a consequence, Simon used these infinite dimensional \L ojasiewicz inequalities, now often called \L ojasiewicz-Simon inequalities, to prove the uniqueness of tangent cones with smooth cross sections of minimal surfaces.
Simon's framework was later applied to the study of compact MCFs by Schulze \cite{S14}. 
By setting up appropriate function spaces on asymptotically self-shrinking hypersurfaces, Chodosh-Schulze \cite{CS} successfully obtained \L ojasiewicz-Simon inequalities on them.
It is worth mentioning that the \L ojasiewicz inequalities obtained by Colding-Minicozzi \cite{CM15,CM19} and Zhu \cite{Z20} were proved without using Simon's theory, and they were able to get explicit \L ojasiewicz inequalities by hand.

The proof of Theorem \ref{prop:uniqueness} is based on the ideas of Chodosh-Schulze \cite{CS}.
The first main idea is, as mentioned above, to construct some function spaces to replace the ordinary H\"older spaces so that we can establish the desired elliptic estimates, based on which Simon's techniques \cite{S83} can be applied.
Elements in these new function spaces are entire vector fields with appropriate decay rates.
Establishing good mapping properties among this new function spaces leads to the Simon-\L ojasiewicz inequality for entire vector fields with sufficiently small norms (in the new function spaces).
Nevertheless, in a blowup sequence, the locally smooth convergence only gives us closeness on (probably large) compact parts of the flows.
As a result, we need to show that some \L ojasiewicz-type inequalities can be derived for vector fields defined only over parts of the self-shrinker.
To deal with this localization process, we need to say that the \textit{rough conical scale} (Definition \ref{scale} (2)) improves whenever the flow can be written as the graph of a (small enough) vector field on a large enough compact set.
Based on this, we can extend the graphicality up to the \textit{shrinker scale} (Definition \ref{scale} (1)).
This allows us to obtain the desired localized \L ojasiewicz-Simon inequality, leading to the uniqueness of tangent flows.

The main goal of this note is to implement the plan above for higher-codimensional MCFs.
It is known that extending results in the hypersurface setting to the higher-codimensional one is not always direct.
The main difficulties include the complexity of curvature tensors and the lack of certain maximum principles.
In our case, most of the work in \cite{CS} can be generalized to the higher-codimensional setting by standard extensions.
However, there are still some places where we need new arguments, and they are all related to the use of certain maximum principles in the hypersurface case.
For example, when proving the extension result in the hypersurface case, Chodosh-Schulze used a barrier principle to show a curvature estimate.
This, in general, does not apply to a MCF of submanifolds of higher codimensions.
To deal with this, we combine the asymptotic conicality and White's curvature estimate to obtain the desired curvature estimate (cf. Proposition \ref{exp}).
This finally leads to the extension step that is necessary for us to utilize the localized \L ojasiewicz-Simon inequality.

This paper is organized as follows. 
In Section \ref{sec:1} and \ref{sec:2}, we give some preliminary notations and results for asymptotically conical self-shrinkers and their normal bundles.
In Section \ref{sec:3}, we talk about entire \L ojasiewicz-Simon inequalities and localized \L ojasiewicz-Simon inequalities. 
In Section \ref{sec:4}, the uniqueness of higher-codimensioanl asymptotically conical tangent flow of MCF is proven.

\section*{Acknowledgement}
The authors are very grateful to Prof. Tobias Colding and Prof. William Minicozzi for their constant support and several inspiring conversations. 
They would also like to thank Prof. Jonathan Zhu for clarifying some results in his paper, thank him, Prof. Otis Chodosh, Prof. Yng-Ing Lee and Prof. Mu-Tao Wang for their interests, and thank Zhihan Wang for pointing out some imprecise statements in an early version.
During the project, Lee was partially supported by NSF Grant DMS 2005345 and Zhao was supported by NSF Grant DMS 1812142 and NSF Grant DMS 1811267.

\section{\bf Geometric and analytic preliminaries}\label{sec:1}

Most of the properties and definitions in this section are similar to the ones in \cite{CS}. 
We mention the parts that we will need or that require some slight modifications in the higher-codimensional case.

\subsection{Estimate on ends}
Consider an asymptotic conical self-shrinker $\Sigma^n\sbst \bb R^{n+k}$ and assume that the asymptotic cone 
\begin{align*}
    \mathcal{C} = \lim\limits_{t\to 0}\sqrt{-t}\Sigma
\end{align*}
is the cone over $\Gamma^{n-1}\sbst S^{n+k-1}.$ 
The metric on $\mathcal{C}$ is $g_{\mathcal{C}}=dr\otimes dr+r^2g_{\Gamma}$. Then from the smooth convergence we know the following decay of curvatures (cf. \cite{W14,CS}).

\begin{Lem}\label{sff}
 There exist $C>0$ and $R_1 \geq 0$ such that for $x\in \Sigma \setminus B_{R_1}$ and $i\ge 0,$
 \begin{align*}
     |\nabla^iA(x)|\leq C |x|^{-i-1}.
 \end{align*}
\end{Lem}

From this lemma and the self-shrinker equation $\mathbf{H}=\frac{x^{\perp}}{2}$, we have an improved estimate
$$
|\nabla^ix^{\perp}|\leq C |x|^{-i-1}.
$$
We list some other estimates we will use later.
Their proofs are similar to those in the hypersurface case.

\begin{Lem}\label{graph}
For $R>0$ sufficiently large, there is $W\in \Gamma(N(\cC\setminus B_R(0)))$ so that 
\begin{align*}
     \gr W:=\{p+W(p)\,:\,p\in \cC\setminus B_R(0)\}\sbst\Sigma
\end{align*}
parametrizes $\Sigma$ outside a compact set. 
The section $W$ satisfies 
\begin{align}
    |W|=O(r^{-1})\text{ and }
    |\nabla^{(j)}W|=O(r^{-1-j+\eta})\label{west}
\end{align}
as $r\to \infty$ for any $\eta>0$ and $j\geq 1$. 
Moreover, the radial derivatives satisfy sharper relations $|\na_{\partial_ r}^{j}W|=O(r^{-1-j})$.
\end{Lem}

\begin{Cor}\label{sharp-metric}
For large $R>0$ and by denoting the corresponding parametrization of $\Sigma\setminus B_R(0)$ using $\cC$ as $F$, we have the following improved estimates on the induced metric 
\[
F^*g_{\Sigma} = dr\otimes dr + r^{2}g_{\Gamma} + h 
\]
where $h$ is a symmetric $(0,2)$-tensor on $\cC\setminus B_{R}(0)$ satisfying $|h| = O(r^{-2})$ and $|\nabla^{(j)}h| = O(r^{-2-j+\eta})$ as $r\to\infty$, for all $j\geq 1$ and $\eta>0$. 
\end{Cor}

The idea of Corollary \ref{sharp-metric} is the same as the one in \cite{CS}.
If we write $F: \cC\setminus B_{R}(0)\to\Sigma$, $F(p) = p + W(p)$ and $W=W^\beta\N_{\beta}$ where $\N_\beta$'s form a local orthonormal basis of $N\cC,$
then we know that
\begin{align*}
 r^{-1}(\partial_{\omega_{i}} W)(p)= r^{-1}\pr{\partial_{\omega_{i}}W^\beta}\N_{\beta}+r^{-1}W^\beta\partial_{\omega_{i}}\N_{\beta}.
\end{align*} 
Using a rotation as in Lemma \ref{graph}, we can  compute (using the fact that $A_{\cC}(\partial_{r},\cdot) = 0$)
\begin{align*}
\partial_{r}F = \partial_{r} + (\partial_{r}W)(p)\text{ and } 
r^{-1}\partial_{\omega_{i}} F = r^{-1} \partial_{\omega_{i}} + r^{-1}(\partial_{\omega_{i}}W^\beta)\N_{\beta} + r^{-1}W^\beta\partial_{\omega_{i}} \N_{\beta}.
\end{align*}
Since we have $|W|=O(r^{-1}),$ $|\partial_r W|=O(r^{-2}),$ $r^{-1}|\partial_{\omega_i} W|=O(r^{-2+\eta}),$ $\pair{r^{-1} \partial_{\omega_{i}}\N_{\beta},\,\,\partial_r},$ and $|\langle r^{-1}\partial_{\omega_{i}}\N_{\beta},\,\,\nu_{\cC,k} \rangle|\leq C,$ we have $|h| = O(r^{-2}).$
The higher derivative estimates follow from interpolation, as in Lemma \ref{graph}. 
We also have the following estimates for the normal bundle of $\Sigma$.

\begin{Lem}\label{normal}
We have $|\nabla^{(j)}(r\partial_{r}W - W)|(p) = O(r^{-1-j})$ for any $j\geq 0,$
\begin{align*}
A_{\Sigma}(\partial_{r}F,\partial_{r}F) & = O(r^{-3}),\\
A_{\Sigma}(\partial_{r}F,r^{-1}\partial_{\omega_{i}}F) & = O(r^{-3}),\text{ and }\\
A_{\Sigma}(r^{-1}\partial_{\omega_{i}}F,r^{-1}\partial_{\omega_{j}}F) & = A_{\cC}(r^{-1}\partial_{\omega_{i}},r^{-1}\partial_{\omega_{j}}) +O(r^{-3+\eta})
\end{align*}
as $r\to\infty$, and $|\nabla^{(k)}_{\cC} (A_{\Sigma}\circ F - A_{\cC})| = O(r^{-3-k+\eta})$ for any $\eta>0$ and $k\geq 1$. 
Moreover, the vector field $V : = \mathbf{proj}_{T\Sigma}F(p) - r\partial_{r} F$ is tangent to $\Sigma$ and satisfies $|V| = O(r^{-1})$, $|\nabla^{(k)}V| = O(r^{-1-k+\eta})$ for $\eta>0$. And the radial derivative satisfies
\[
\vec{x}\cdot \nabla_{\Sigma} f =r(\partial_{r}f)^\perp + \alpha_{3}\cdot\nabla_{g_{\cC}} f,
\]
where $|\alpha_{3}| = O(r^{-1})$ and $|\nabla^{(j)}\alpha_{3}| = O(r^{-1-j+\eta})$ for $\eta>0$ and $j\geq 1$, as $r\to\infty$. 
\end{Lem}

\subsection{Mapping properties between weighted H\"older spaces} 
Next, we define some relevant weighted H\"older spaces of normal sections as in \cite{CS}.
The spaces defined in \cite{CS} were constructed based on the ones in \cite{KKM}.
We will then mention a Schauder-type estimate for the stability operator for the self-shrinker equation.
We will follow our notations in the previous subsection.

From now on, we fix a cutoff function $\chi : [0,\infty) \to [0,1]$ so that $\supp \chi \sbst [R,\infty)$, $\chi \equiv 1$ in $[2R,\infty)$, and $|\nabla^{j} \chi| \leq C R^{-j}$ for $j\geq 1$ and $C$ independent of $R$ sufficiently large. 
This allows us to define the primary H\"older spaces.

\begin{Def}
Let $\gamma\in\bb R$ and $\alpha\in(0,1).$

\noindent (1) (Homogeneously weighted H\"older spaces)
For a normal vector field $V$ over $\Sigma,$ we define a norm 
$$
\Vert V \Vert_{0;-\gamma}^{\text{hom}} : = \sup_{x\in\Sigma} \tilde r(x)^{\gamma} |V(x)|
$$
and a semi-norm
$$
[ V\, ]_{\alpha;-\gamma-\alpha}^{\text{hom}} : = \sup_{x,y\in \Sigma} \frac{1}{\tilde r(x)^{-\gamma-\alpha} + \tilde r(y)^{-\gamma-\alpha}} \frac{|V(x)-\mathbf{proj}_{N_x\Sigma}V(y)|}{|x-y|^{\alpha}}.$$
Then we let $C^{0,\alpha}_{\text{hom},-\gamma}(\Sigma)$
to be the set of normal vector fields $V$ such that
$$\Vert V \Vert_{0,\alpha;-\gamma}^{\text{hom}} : = \Vert V \Vert_{0;-\gamma}^{\text{hom}} + [ V ]_{\alpha;-\gamma-\alpha}^{\text{hom}}< \infty. $$
Similarly, we define $C^{2,\alpha}_{\text{hom},-\gamma}(\Sigma)$
to be the set of $V$ such that
$$\Vert V \Vert_{2,\alpha;-\gamma}^{\text{hom}} = \sum_{j=0}^{2} \Vert (\nabla_{\Sigma})^{(j)}V \Vert_{0,\alpha;-\gamma}^{\text{hom}}<\infty.$$
(2) (Anisotropically weighted H\"older spaces) We define $C^{2,\alpha}_{\text{an},-1}(\Sigma)$
to be the space of 
$V \in C^{2,\alpha}_{\text{hom},-1}(\Sigma)$ such that 
$$\Vert V \Vert_{2,\alpha;-1}^{\text{an}} : = \Vert V \Vert_{2,\alpha;-1}^{\text{hom}} + \Vert \vec{x}\cdot \nabla_{\Sigma} V \Vert_{0,\alpha;-1}^{\text{hom}}<\infty.$$

\noindent (3) (Cone H\"older spaces) We define 
$\CS 0 : = C^{0,\alpha}_{\textrm{hom},-1}(\Sigma)$ 
and
$$\CS 2 : = C^{2,\alpha}(\Gamma) \times C^{2,\alpha}_{\text{an},-1}(\Sigma).$$
An element 
$(v,V) \in \cC\cS_{-1}^{2,\alpha}(\Sigma)$ 
will be considered as a section on $\Sigma$ given by
$$U = U_{(v,V)}(r,\omega) = \chi(r) (v(\omega))^\perp r + V(r,\omega)$$
for $r\ge R$, and $u=V$ otherwise. 
We will write $U=U_{(v,V)}$ if everything is clear. 
Finally, we define the norm
$$\| U \|_{\CS 2} : = \inf_{U = \chi(r) (v(\omega))^\perp r + V(r,\omega)}\Vert v \Vert_{C^{2,\alpha}(\Gamma)} + \Vert V \Vert_{2,\alpha;-1}^{\text{an}}.$$
\end{Def}

Next, we will prove a Schauder-type estimate for the operator $\cL^\perp_{\frac 12},$ defined by
$$\cL^\perp_{\frac12} V  : = \Delta_{\Sigma}^\perp V - \frac 12 \cdot \np_{x^T} V + \frac 12 V.$$ 
An important ingredient is a vector-valued version of Brandt's non-standard interior Schauder estimate.

\begin{Pro}\label{intSch}
	Suppose $B_2\in \bb R^n$ and there are $C^{0,\alpha}$ functions $a_{ij},b_i^{kl},c^{kl},u^k,f\colon B_2\times [-2,0]\to\bb R$ such that
	$$\frac{\bd}{\bd t}u^k
	= a_{ij} D_{ij}^2 u^k
	+ b_i^{kl} D_iu^l
	+ c^{kl}u^l
	+ f
	$$
	for all $k=1,\cdots,n.$ If the system is uniformly parabolic in the sense that
	$$a_{ij}\xi_i\xi_j\ge \lambda|\xi|^2$$
	for some $\lambda>0$ and 
	$$\sup_{t\in[-2,0]}
	\pr{
		\|a_{ij}(\cdot,t)\|_{C^{0,\alpha}(B_1)}
		+ \|b_{i}^{kl}(\cdot,t)\|_{C^{0,\alpha}(B_1)}
		+ \|c^{kl}(\cdot,t)\|_{C^{0,\alpha}(B_1)}
	}
	\le \Lambda$$
	for some $\Lambda>0,$ then for $T\in (-1,0],$
	$$
	\sup_{t\in[-1,T]} \|u(\cdot,t)\|_{C^{2,\alpha}(B_1)}
	\le C\sup_{t\in[-2,T]} \pr{
		\|u(\cdot,t)\|_{C^{0,\alpha}(B_2)}
		+ \|f(\cdot,t)\|_{C^{0,\alpha}(B_2)}
	}
	$$
	for some $C=C(n,\alpha,\lambda,\Lambda)$ where $u=(u^1,\cdots,u^m).$
\end{Pro}

\begin{proof}
	By the non-standard Schauder estimate for parabolic equations established in \cite{Brandt} (see also \cite[Theorem 3.6]{CS}), for any $k$ and $t\in [-2,0],$ we have
	$$\|u^k(\cdot,t)\|_{C^{2,\alpha}(B_1)}
	\le \|u^k(\cdot,t)\|_{C^{0,\alpha}(B_2)}
	+ \|b_i^{kl}(\cdot,t) D_iu^l(\cdot,t)
	+ c^{kl}(\cdot,t) u^l(\cdot,t)
	+ f(\cdot,t)\|_{C^{0,\alpha}(B_2)}.
	$$
	Then the proposition follows by rearrangement after combining this and the standard estimate
	$$\|h\|_{C^2(B_1)}
	\le \varepsilon [D^2 h]_{\alpha, B_2}
	+ C_1(\alpha)\varepsilon^{-C_2(\alpha)}\|h\|_{C^0(B_2)}$$
	for any $\varepsilon>0$ and any $C^{2,\alpha}$ function $h\colon B_2\to\bb R.$
\end{proof}

We would like to use this estimate and the diffeomorphism generated by the position vector field to obtain the Schauder-type estimate we would like.
Similar as in \cite{CS}, we consider the time dependent vector field $X_t=\frac{1}{-2t}x^T$ and define $\Phi_t:\Sigma\to\Sigma$ as $\frac{\partial}{\partial t}\Phi_t=X_t\circ \Phi_t$ with $\Phi_{-1}=Id$ with $\hat{g_t}=(-t)\Phi_t^*g_\Sigma$. 
Then we know that if $F:\Sigma\to \bb R^{n+k}$ is the embedding of $\Sigma$ in $\bb R^{n+k}$, $\hat{F_t}=\sqrt{-t}(F\circ \Phi_t):\Sigma\to\bb R^{n+1}$ is a MCF and $\hat{g_t}=\hat{F_t}^*g_{\bb R^{n+k}}$.

We can parameterize the end of $\Sigma$ using $\cC\setminus B_r(0)$ as we did before. 
We denote the parametrization as $F:\cC\setminus B_r(0),\,p\to p+W(p)$. 
Then we define
\begin{align*}
    \tilde{\Phi_t}=F^{-1}\circ \Phi_t\circ F
\end{align*}
and 
\begin{align*}
     \phi_t:(R,\infty)\times \Gamma\to (R,\infty)\times \Gamma,\,(r,\omega)\mapsto ((-t)^{-1/2}r,\omega)
     \text{ for }t\in [-1,0).
\end{align*}
Similarly as in \cite{CS}, we know that for $t\in[-1,0)$ and $r$ sufficiently large,
\begin{align*}
d_{g_{\cC}}(\tilde{\Phi_t}(t,\theta),\phi_t(r,\theta))\lesssim\frac{1}{\sqrt{-t}\cdot r},
\end{align*}
and in the flat cylindrical metric $dr^2+g_\Gamma,$ we have 
\begin{align*}
|D^{(j)}(\tilde{\Phi_t}-\phi_t)|(r,\theta)\lesssim\frac{1}{\sqrt{-t}\cdot r^{1+j-\eta}}.
\end{align*} 
 
In the higher-codimensional case, the stability operator for the self-shrinker equation is\footnote{One can see \cite{ALW}, \cite{AS} or \cite{LL} for the higher-codimensional generalization of the one defined in \cite{CM12}.} 
\begin{align*}
	L  &= \Delta^\perp_\Sigma   
 - \frac{1}{2}\np_{x^T}  + \frac{1}{2} + \sum_{k,\ell}  \,  \langle  \cdot   , A_{k \ell}  \rangle \, A_{ k \ell} = \cL^\perp_{\frac{1}{2}}+ \sum_{k,\ell}  \,  \langle  \cdot   , A_{k \ell}  \rangle \, A_{ k \ell} \, .
\end{align*}
Then for a normal section $V$ such that $\cL^\perp_{\frac{1}{2}}V+\sum_{k,\ell}  \,  \langle  V   , A_{k \ell}  \rangle \, A_{ k \ell}=E$, if we define $a\colon N\Sigma\to N\Sigma$ by $a(V)=\sum_{k,\ell}  \,  \langle V, A_{k \ell}  \rangle \, A_{ k \ell}$, we would know that $\|a\|_{C^{0,\alpha}(B_1(x))}=O(|x|^{-2})$. 
Following this, consider
\begin{align*}
\hat{V}(x,t)=\sqrt{-t}\,V(\Phi_t(x)),\,\,\,\hat{E}(x,t)=\frac{1}{\sqrt{-t}}E(\Phi_t(x))\text{ and } \hat{a}(x,t)=\frac{1}{-t}a(\Phi_t(x)).
\end{align*}
Then we have 
\begin{align*}
\frac{\partial \hat{V}}{\partial t}-\Delta^\perp_{\hat{g_t}}\hat{V}-\hat{a}(\hat{V})=\hat{E}.
\end{align*}
Using these, we have the corresponding claim in higher codimension. 
\begin{Pro}\label{Sc}
Consider $a\colon N\Sigma \to N\Sigma$ with $\Vert a \Vert_{C^{0,\alpha}(B_{1}(x))} = O(|x|^{-2})$ for $x\in\Sigma$ with $|x|\to\infty$, i.e., $a \in C^{0,\alpha}_{\textnormal{hom};-2}(\Sigma)$. 
There exists $C=C(\Sigma,a)$ such that if $V \in C^{2,\alpha}_{\textnormal{loc}}(\Sigma) \cap C^{0}_{\textnormal{hom};+1}(\Sigma)$ satisfies $\cL^\perp_{\frac 12} V + aV  \in \CS 0$, then $V\in \CS 2$ and we have the estimate
\[
\Vert V \Vert_{\cC\cS_{-1}^{2,\alpha}(\Sigma)} \leq C\left( \Vert V \Vert_{C^{0}_{\textnormal{hom},+1}(\Sigma)} + \Vert \cL^\perp_{\frac 12} V + aV \Vert_{\cC\cS_{-1}^{0,\alpha}(\Sigma)} \right).
\] 
\end{Pro}
 
The proof is essentially the same as in the codimension one case based on the discussion above and Proposition \ref{intSch}.
Note that in the higher-codimensional case, we set 
\begin{align*}
     W:=
     r(\partial_r V)^\perp-V
     =-2E+2\Delta^\perp V+2a(V)-(\vec{x}\cdot \np_{\Sigma} V-r(\partial_rV)^\perp)
\end{align*}
and the estimate comes from Lemma \ref{normal}.

\section{\bf Estimates on weighted spaces}\label{sec:2}
In this section, we establish some linear estimates on weighted Sobolev spaces, extending those in the hypersurface case in \cite{CS} to the higher-codimensional case.

Let $\rho(x):=(4\pi)^{-n/2} e^{-|x|^2/4}$ be the $n$-dimensional backward heat kernel. 
For any normal vector field or any tensor $V$ over $\Sigma,$ we define 
$$\|V\|_{L^2_W(\Sigma)}^2
:= \int_\Sigma |V|_g^2 \cdot \rho~dx,$$
and let $L^2_W$ be the space of normal vector field $V$ over $\Sigma$ such that $\|V\|_{L^2_W(\Sigma)}^2<\infty.$
Using this, we consider 
$$\|V\|_{H^k_W(\Sigma)}^2
:= \sum_{i=0}^k \|(\np)^i V\|_{L^2_W}^2,$$
and let $H^k_W(\Sigma)$ be the associated Sobolev space of normal vector fields.
When it is clear, we will omit the symbol $\Sigma$ and just write $\|V\|_{L^2_W}$ and $\|V\|_{H^k_W}$

The following Sobolev inequality was due to Ecker \cite{E00} in the hypersurface case.
\begin{Pro}\label{ecker}
	For $V\in H^1_W(\Sigma),$ we have
	$$\int_\Sigma |V|^2\cdot |x|^2 \rho~ dx
	\le 4\int_\Sigma \pr{n|V|^2 + 4|\np V|^2} \rho~ dx.$$
\end{Pro}
\begin{proof}
	Consider a smooth normal vector field $V$ over $\Sigma$ with compact support. Then the first variation formula implies
	\begin{align*}
	-\int_\Sigma \pair{|V|^2\rho(x) x,\h}dx
	& = \int_\Sigma \text{div} \pr{|V|^2\rho(x)x}dx\\
	& = \int_\Sigma \pr{
	n|V|^2
	+ \pair{\n|V|^2,x}
	- \frac 12|V|^2\cdot |x^T|^2
	}\rho~dx.
	\end{align*}
	Since $\Sigma$ is a self-shrinker (that is, $\h=-\frac{x^\perp}{2}$), we derive
	\begin{align*}
	\frac 12 \int_\Sigma |V|^2\cdot |x|^2 \rho~dx
	& =  \int_\Sigma \pr{
		n|V|^2
		+ 2|V|\pair{\n|V|,x}
	}\rho~dx\\
	& \le \int_\Sigma \pr{
		n|V|^2
		+ 4|\np V|^2 + \frac 14 |x|^2|V|^2
	}\rho~dx
	\end{align*}
	by the Cauchy-Schwarz inequality and the Kato inequality. 
	Then the inequality for $V$ follows after rearrangement. 
	This implies the inequality for $V\in H^1_W(\Sigma)$ using an approximation argument.
\end{proof}

\begin{Lem}\label{DV}
	For $V\in H^2_W(\Sigma),$ we have
	$\|\np V\|_{L^2_W}^2 \le \|\dlp V\|_{L^2_W} \|V\|_{L^2_W}.$
    In particular, the inclusion $H^1_W(\Sigma)\hookrightarrow L^2_W(\Sigma)$ is compact.
\end{Lem}
\begin{proof}
	As above, it suffices to prove the inequality for smooth $V$ with compact support. Using integration by parts, we have
	\begin{align*}
	0 
	& = \int_\Sigma \dl |V|^2 \rho~dx\\
	& = 2\int_\Sigma\pr{
	\pair{\D^\perp V,V}
	+ |\np V|^2
	- \frac 12 \pair{\np_{x^T}V,V}
	} \rho~dx\\
	& = 2\int_\Sigma \pr{
	\pair{\dlp V,V}
	+ |\np V|^2
	} \rho~dx.
	\end{align*}
	As a result, the first conclusion follows from the Cauchy-Schwarz inequality. The same argument in \cite[lemma 3.12]{CS} implies the second conclusion.
\end{proof}


\begin{Pro}\label{DDV}
	For $V\in H^2_W(\Sigma),$ we have
	$$\|V\|_{H^2_W}^2 \le C \pr{
	\|\dlp V\|_{L^2_W}^2
	+ \|V\|_{L^2_W}^2
	}
	$$
	for some $C=C(\Sigma).$
\end{Pro}
\begin{proof}
	As above, it suffices to prove the inequality for smooth $V$ with compact support.
	Around a point $p\in \Sigma,$ take normal coordinates $e_1,\cdots,e_n$ such that $\n_{e_i}e_j(p)=0.$ 
    At $p,$ we have
	\begin{align*}
	\D |\np V|^2
	& = \D \pr{\sum_i \pair{\np_{e_i}V,\np_{e_i}V}}\\ 
	& = 2 \sum_{i,k} \pair{\np_{e_k}\np_{e_k}\np_{e_i} V, \np_{e_i}V}
	+ 2 \sum_{i,k} \pair{\np_{e_k}\np_{e_i} V, \np_{e_k}\np_{e_i} V}\\
	& = 2\sum_{i,k}\pair{
	\np_{e_k}\np_{e_i}\np_{e_k} V
	+ \n_{e_k}\pr{R^\perp(e_i,e_k)V},
	\np_{e_i} V
	}
	+ 2|(\np)^2V|^2\\
	& = 2 \sum_{i,k} \pair{
	\np_{e_i}\np_{e_k}\np_{e_k} V
	+ R^\perp(e_i,e_k) \np_{e_k} V
	+ \n_{e_k}\pr{R^\perp(e_i,e_k)V},
	\np_{e_i} V
	}
	+ 2|(\np)^2V|^2\\
	& = 2\pair{\np\D^\perp V,\np V}
	+ 2|(\np)^2V|^2
	+ 4 \pair{R^\perp(e_i,e_k) \np_{e_k} V, \np_{e_i}V}
	+ 2\pair{\pr{\n_{e_k}R^\perp}(e_i,e_k)V,\np_{e_i}V},
	\end{align*}
	where the normal curvature 
	$$R^\perp(X,Y)W
	=\np_Y\np_X W - \np_X\np_Y W + \np_{[X,Y]}W$$
	is defined for any tangent vectors $X,Y$ and any normal vector $W.$ 
    Therefore, using Lemma \ref{sff} and the Ricci equation
	$$\pair{R^\perp(X,Y)W,Z}
	= \sum_{k=1}^n \pr{
	\pair{A(X,e_k),W}\pair{A(Y,e_k),Z}
	- \pair{A(Y,e_k),W}\pair{A(X,e_k),Z}
	},
	$$
	we can estimate
	\begin{equation*}
	\D |\np V|^2
	=  2\pair{\np\D^\perp V,\np V}
	+ 2|(\np)^2V|^2
	+ O(|V|^2+|\np V|^2).
	\end{equation*}
	Therefore, we have
	\begin{align*}
	&~~~~\dl |\np V|^2\\
	& = 2\pair{\np\dlp V,\np V}
	+ 2|(\np)^2V|^2
	- \frac 12\n_{x^T} |\np V|^2
	+ \pair{\np\np_{x^T}V, \np V}
	+ O(|V|^2+|\np V|^2)\\
	& = 2\pair{\np\dlp V,\np V}
	+ 2|(\np)^2V|^2
	+ \sum_i \pr{-\pair{\np_{x^T}\np_{e_i}V,\np_{e_i}V}
	+ \pair{\np_{e_i}\np_{x^T}V, \np{e_i} V}}\\
	& + O(|V|^2+|\np V|^2)\\
	& = 2\pair{\np\dlp V,\np V}
	+ 2|(\np)^2V|^2
	+ \sum_i \pair{R^\perp(x^T,e_i)V,\np_{e_i}V}
	+ O(|V|^2+|\np V|^2)\\
	& = 2\pair{\np\dlp V,\np V}
	+ 2|(\np)^2V|^2
	+ O(|V|^2+|\np V|^2)
	\end{align*}
	using Lemma \ref{DV} and the Ricci equation again. 
	Hence, after integrating over $\Sigma$ and using integration by parts, we get 
	\begin{align*}
	0 
    = -2 \|\dlp V\|_{L^2_W}^2
	+ 2\|(\np)^2V\|_{L^2_W}^2
	+ O\pr{\|V\|_{L^2_W}^2 
    + \|\np V\|_{L^2_W}^2}.
	\end{align*}
	The proposition then follows after combining this with lemma \ref{DV}.
\end{proof}

Now we can derive the existence of weak solutions to the stability operator $L.$ 
To this end, recall that we define
$$LV 
:= \dlp V + \sum_{i,j}\pair{V,A_{ij}} A_{ij} + V
=\D^\perp V - \frac 12 \n_{x^T}^\perp V 
+ \sum_{i,j}\pair{V,A_{ij}} A_{ij} + V
$$
for $V\in N\Sigma$, and
$$B_\gamma(U,V)
= \int_\Sigma \pr{\pair{\np U, \np V}
- \pr{\gamma + 1}\pair{U,V}
- \sum_{i,j} \pair{U,A_{ij}} \pair{V,A_{ij}}
} \rho~ dx,$$
which is the bounded bilinear form associated to $-(L+\gamma).$ 
By taking large enough $\gamma$ such that $\gamma\ge \sup_\Sigma |A|^2 + 2,$ we have
$$\|V\|_{H^2_W}^2
\le \abs{B_\gamma (V,V)}.
$$
Using this estimate, the standard Fredholm alternative, the Lax-Milgram theorem, Lemma \ref{DV} and Proposition \ref{DDV}, we obtain the following result.

\begin{Thm}\label{kerL}
	The space $\Ker L$ of the weak solutions to $LV=0$ is finite-dimensional. Moreover, for $U\in L^2_W(\Sigma),$ $LV=U$ has a weak solution in $H_W^1(\Sigma)$ if and only if $U$ is $L^2_W$-orthogonal to $\Ker L,$ and if $V$ is orthogonal to $\Ker L$ also, then 
    $\|V\|_{H^2_W}\lesssim \|U\|_{L^2_W}.$
\end{Thm}

 Finally we generalize \cite[Lemma 3.15]{CS} to sections of the normal bundle.
\begin{Pro}
	\label{C0toC0hom1}
	There exists $C>0$ such that for any section $X \in L^{2}_{W}(\Sigma) \cap C^{0}(\Sigma)$ of the normal bundle, 
    if ${V} \in H^{1}_{W}(\Sigma)$ satisfies $L V=X,$ then ${V} \in C^{0}_{\textnormal{hom};+1} (\Sigma)$ and for $R$ sufficiently large,
	\[
	\Vert V \Vert_{C^{0}_{\textnormal{hom};+1} (\Sigma)} 
    \le C\pr{ \Vert X \Vert_{C^{0}(\Sigma)} + \Vert V \Vert_{C^{0}(\Sigma\cap B_{R}(0))}}.
	\]
\end{Pro}

\begin{proof}
	Define $\varphi(x)=\alpha|x|$ with $\alpha$ to be specified later, and consider the section $U=\frac{V}{\varphi}$. We have that 
	\begin{align*}
	LU
	= L\pr{\frac{V}{\varphi}}
	= \frac{LV}{\varphi} 
	+ 2\na^{\perp}_{\na\frac{1}{\varphi}} V 
	+ \pr{ \Delta\pr{\frac{1}{\varphi}}
	- \frac{1}{2} \na_{{x}^T} \pr{\frac{1}{\varphi}}} V.
	\end{align*}
	Since we know 
	$$
	\na\pr{\frac{1}{\varphi}}
	= -\frac{\alpha x^T}{\varphi^2|x|}
	$$
	and
	$$
	\Delta\pr{\frac{1}{\varphi}}
	= -\alpha\pr{\frac{n-\la x^\perp,\h\ra}{\varphi^2|x|}}+\frac{2\alpha^2|x^T|^2}{|x|^2\varphi^3}+\frac{\alpha|x^T|^2}{\varphi^2|x|^3},
	$$
	we can derive
	\begin{align*}
	\Delta\pr{\frac{1}{\varphi}}
	- \frac{1}{2} \na_{{x}^T} \pr{\frac{1}{\varphi}}
	= \frac{\alpha|x|}{2\varphi^2}
	- \frac{n\alpha}{\varphi^2|x|}
	+ \frac{2\alpha^2|x^T|^2}{|x|^2\varphi^3}
	+ \frac{\alpha|x^T|^2}{\varphi^2|x|^3}
	= \frac{1}{2\alpha|x|}
	- \frac{n}{\alpha|x|^3}
	+ \frac{3|x^T|^2}{|x|^5\alpha}.
	\end{align*}
	Combining these with $|A|=O(r^{-1})$, we can estimate
	\begin{align}
	\la \cL_{\frac{1}{2}} U,\,U\ra
	= \la LU,\, U\ra-|AU|^2
	&= \pair{ \frac{X}{\varphi} + 2\na^\perp_{\na\frac{1}{\varphi}}(U\varphi) + \frac{1}{2}U,\,U }
	+ O(r^{-2})|U|^2\notag\\
	&= \frac{1}{2}|U|^2
	+ \pair{ \frac{X }{\varphi},\,U}
	- \frac{2|U|^2|\na\varphi|^2}{\varphi^2} + 2\varphi\,\na^\perp U\pr{U ,\,\na \frac{1}{\varphi}} + O(r^{-2})|U|^2.
	\end{align}
	
	Next, we take a cutoff function $\eta:[0,+\infty)\to\R$  such that $\eta'\geq0$ and $\eta(x)=\left\{\begin{array}{cc}
	0,&\,0<x<\frac{1}{2}  \\
	1,&\,x>1 
	\end{array}\right.$. 
	Then we have
	\begin{align}\label{nong}
	\na^\perp U(U,\na \eta^2(|U|))=2\eta\eta'  \na^\perp U(U,\na|U|)=2\eta\eta'  \na^\perp U(U,\frac{\na^\perp U(U)}{|U|})=2\eta\eta'  \frac{|\na^\perp U(U)|^2}{|U|}\geq 0.
	\end{align}  
	We set
	\begin{align}
    \alpha = 2R^{-1} \sup\limits_{\Sigma\cap B_R(0)} |V| 
    + 8R^{-1} \sup\limits_{\Sigma\cap B_R(0)} |X|.
	\end{align}
	This implies $\alpha R\geq 2\sup\limits_{\Sigma\cap B_R(0)}|V|$,\, and hence, we have  
	\begin{align}\label{oneside}
	\int_{\Sigma\setminus B_R(0)}\eta^2(|U|) \la \cL_{\frac{1}{2}} U,\,U\ra\rho dx&
    = \int_{\Sigma\setminus B_R(0)} \pr{ \eta^2(|U|) \pr{-|\na^\perp U|^2+\frac{1}{2}|U|^2} 
    -\na^\perp U(U,\na \eta^2(|U|))}\rho dx\notag\\ 
    &\leq \int_{\Sigma\setminus B_R(0)} \eta^2(|U|) \pr{-|\na^\perp U|^2+\frac{1}{2}|U|^2}
	\end{align}
	by \eqref{nong}. 
	On the other hand, as $|A|=O(r^{-1})$ and $\alpha R\geq 8\sup\limits_{\Sigma}|X |$, we have
	\begin{align}\label{otherside}
	\eta^2(|U|) \la \cL_{\frac{1}{2}} U,\,U\ra
    = &  \eta^2(|U|)\pr{\frac{1}{2}|U|^2+ \pair{\frac{X}{\varphi} ,\,U}
    + 2\varphi\na^\perp U \pr{U ,\,\na \frac{1}{\varphi}} 
    + O(R^{-2})|U|^2}\notag\\
    \geq& \eta^2(|U|) \pr{\frac{1}{8}|U|^2+O(R^{-2}){|\na^\perp U|^2}+O(R^{-2})|U|^2}.
	\end{align}
	
	Note that if we apply the Sobolev inequality (Proposition \ref{ecker}) to $\eta(|U|)U,$ we have
	\begin{align*}
	&~~~~R^2 \int_{\Sigma\setminus B_R}
	\eta^2|U|^2\rho~ dx\\
	& \le 4n\int_{\Sigma\setminus B_R}
	\eta^2|U|^2\rho~ dx
	+ 16\int_{\Sigma\setminus B_R}
	\abs{\np(\eta U)}^2 \rho~ dx\\
	& \le 4n\int_{\Sigma\setminus B_R}
	\eta^2|U|^2\rho~ dx
	+ 32\int_{\Sigma\setminus B_R}
	\eta^2\abs{\np U}^2\rho~ dx
	+ 32\int_{\Sigma\setminus B_R}
	\pr{\eta'(|U|)}^2\abs{\np U}\cdot |U|^2\rho~ dx\\
	& \le 4n\int_{\Sigma\setminus B_R}
	\eta^2|U|^2\rho~ dx
	+ 32\int_{\Sigma\setminus B_R}
	\eta^2\abs{\np U}^2\rho~ dx
	+ 32\sup\eta'\cdot \int_{\Sigma\setminus B_R}
	\abs{\np U}^2\rho~ dx
	\end{align*}
	since $\eta'\neq 0$ only when $|U|\in (1/2,1).$
	Consequently, combining \eqref{oneside}, \eqref{otherside} and Proposition \ref{ecker}, we get that 
	\begin{align*}
	&~~~~\frac{R^2-4n}{32}
	\int_{\Sigma\setminus B_R}
	\pr{1+O\pr{R^{-2}}} \eta^2|U|^2\rho~ dx
	+ \int_{\Sigma\setminus B_R}
	\pr{\frac 12+ O\pr{R^{-2}}}\eta^2|U|^2\rho~ dx\\
	& \le \pr{4+O\pr{R^{-2}}} 
	\int_{\Sigma\setminus B_R}
	\abs{\np U}^2\rho~ dx.
	\end{align*}
	Thus, if we take $R$ sufficiently large, we can derive that $\eta^2(|U|) |U|^2=0$, which implies that $|U|\leq 1$. 
    As $U=\frac{V}{\alpha|x|}$ ,this is equivalent to that \begin{align}
	\Vert V \Vert_{C^{0}_{\textnormal{hom};+1}(\Sigma)} \lesssim \Vert X \Vert_{C^{0}(\Sigma)} + \Vert V \Vert_{C^{0}(\Sigma\cap B_{R}(0))}.
	\end{align}
	This proves the desired estimate.
\end{proof}

Combining Proposition \ref{C0toC0hom1} with the Schauder-type estimate (Proposition \ref{Sc}), we obtain the following corollary, which is essential when applying Simon's theory to obtain a \L ojasiewicz inequality.
\begin{Cor}\label{SolveL}
    For $U\in\CS 0,$ there exists a solution $V\in\CS 2$ to $LV=U$ if and only if $U$ is $L^2_W$-orthogonal to $\Ker L\sbst H^1_W(\Sigma).$
\end{Cor}

\section{\bf \L ojasiewicz-Simon inequalities}\label{sec:3}
In this section, we will set up the frameworks that enable us to apply Simon's theory to establish a \L ojasiewicz-Simon inequality for entire vector fields over an asymptotically conical self-shrinker, and localize it to the form we would like to apply later.
We will follow the guidelines of \cite{CS} and \cite{S83}. 

Assume $M$ is a normal graph of $V\in N\Sigma$ over $\Sigma.$ 
Recall that the Gaussian area (or the $F$-functional) and the entropy of $M$ are defined to be
$$F(M)
:= \int_M \rho~ dx
\text{ and }
\lambda(M)
:= \sup_{x_0\in\bb R^{n+k},~t_0>0}
F(t_0 M + x_0),
$$
where we recall that $\rho(x):= (4\pi)^{-n/2}e^{-|x|^2/4}$. 
By \cite[Section 3]{S14}, we know that given a variation $W$ of $V,$
$$
\delta_{W} F(M)
=-\int_\Sigma 
\pair{
	\Pi_{N\Sigma}\pr{\h_M + \frac{x^\perp}{2}}|_{x=y+V(y)},
	W(y)
}
J(y,V,\n^\Sigma V)
\rho(y+V(y))
dy
$$
where $\Pi_{N\Sigma}$ is the projection onto $N\Sigma$ and $J(y,V,\n^\Sigma V) = \text{Jac}(D\exp_y(V))$ is the area element. 
Thus, the linearization of $F$ is 
\begin{equation*}
\M(V)
= J(y,V,\n^\Sigma V)
\rho(y+V(y))
\rho(y)^{-1}\cdot
\Pi_{N\Sigma}\pr{\h_M + \frac{x^\perp}{2}}|_{x=y+V(y)}.
\end{equation*}

To exploit the ideas in \cite{S83}, we need some good properties of the map $\M.$
The proofs are parallel to the ones in \cite{CS} with some standard modifications.

\begin{Lem}\label{propM}
	For $\beta$ small enough depending on $\Sigma,$ the map
	$$\M\colon \{u\in\CS{2}:
	\|u\|_{\CS{2}} <\beta
	\}\to \CS{0}
	$$
	is Fr\'echet differentiable with its derivative at $0$ being $L.$
\end{Lem}

\begin{proof}
	First we prove the continuity of $\M.$ For $V\in\CS{2},$ note that 
	$$
	J(y,V,\n^\Sigma V)
	\rho(y+V(y))
	\rho(y)^{-1}
	= J(y,V,\n^\Sigma V)\cdot
	e^{-\frac{2\pair{V(y),y}+|V(y)|^2}{4}},
	$$
	which is uniformly bounded in $C^{0,\alpha}(\Sigma)$ as $y\to\infty.$ 
	
	For the last term, at a point $y\in\Sigma,$ if we let $F$ be a parametrization of $\Sigma$ near $y,$ then we write $\td F$ to be the parametrization of $M$ near $y+V(y),$ given by
	\begin{equation}\label{tdF}
	\td F(y) = F(y) + V(y)
	= F(y) + V^\beta(y)\N_\beta(y)
	\end{equation}
	where $\{\N_1,\cdots,\N_k\}$ is a local orthonormal basis of $N\Sigma$ near $p.$ Then we have
	$$\bd_i\td F = \bd_i F + \pr{\bd_i V^\beta}\N_\beta - V^\beta A_{ik}^\beta\bd_k F$$
	where $A_{ij}=A_{ij}^\beta\N_\beta$ is the second fundamental form of $\Sigma.$ 
    Therefore, by direct calculations (similar to those used to derive \cite[(2.27)]{W14}), we can see that if we let $\Phi:=\pair{V,A}=(\pair{V,A_{ij}})_{ij},$ then 
	$$N_\beta := -\pr{\id-\Phi}^{-1}\n^\Sigma V^\beta + \N_\beta
	$$
	forms a local orthogonal basis ($\beta=1,\cdots,k$) of $NM$ near $\td F(y).$ 
	Combining this with the self-shrinker equation and the curvature decay (Lemma \ref{sff}), we have
	$$\abs{
		\pair{x,\frac{N_\beta}{|N_\beta|}}
	}
	= \frac 1{|N_\beta|} \pr{V^\beta - \pair{y,\n^\Sigma V^\beta} + O\pr{|y|^{-1}}}
	$$
	at $x=y+V(y).$ This shows that $\M$ is a bounded map, and its differentiability follows from a similar standard argument.
\end{proof}

By Theorem \ref{kerL}, $\Ker L\sbst H^1_W(\Sigma)$ is finite-dimensional. 
Define $\Pi\colon L^2_W(\Sigma)\to L^2_W(\Sigma)\cap\CS{2}$ to be the projection onto $\Ker L$ and $\sN:=\M+\Pi.$
By Lemma \ref{propM}, $\sN$ is Fr\'echet differentiable with its derivative at $0$ being $L+\Pi,$ which is an isomorphism from $\CS{2}$ to $\CS{0}.$ 
Therefore, the inverse function theorem implies that there exists open neighborhoods $W_1\sbst \CS 2$ and $W_2\sbst \CS 0$ of $0$ such that $\sN\colon W_1\to W_2$ is bijective with the inverse $\Psi\colon W_2\to W_1.$ 
We next prove that these maps behave nicely when restricted to these neighborhoods.

\begin{Lem}\label{Mcon}
	After shrinking $W_1$ and $W_2$ if necessary, we can find $C>0$ such that 
	$$\|\M(V_1)-\M(V_2)\|_{L^2_W} \le C\|V_1-V_2\|_{H^2_W}$$
	for $V_1,V_2\in W_1$ and 
	$$\|\Psi(U_1)-\Psi(U_2)\|_{H^2_W}\le C\|U_1-U_2\|_{L^2_W}$$
	for $U_1,U_2\in W_2.$
\end{Lem}
\begin{proof}
	At a point $y\in\Sigma,$ if we let $F$ be a parametrization of $\Sigma$ near $y,$ then we write $\td F$ to be the parametrization of $M$ near $y+V(y),$ given by \eqref{tdF}. We may assume $g_{ij}=\pair{\bd_i F,\bd_j F}=\delta_{ij}$ at $y.$ Then using the calculations in the proof of \cite[lemma 3.18]{Kh}, we have
	\begin{equation*}
	\td g_{ij} 
	:= \pair{\bd_i\td F, \bd_j\td F}
	= \delta_{ij}
	+ \pair{\bd_iV,\bd_jV}
	- 2\pair{V,A_{ij}}
	+ V^\beta V^\gamma A_{ik}^\beta A_{jk}^\gamma
	\end{equation*}
	at $y+V(y)$ where we write $V=V^\beta\N_\beta$ and $A_{ij}=A_{ij}^\beta\N_\beta$ in a local orthonormal basis $\{\N_\beta\}_{\beta=1,\cdots,k}.$ Using this, we can derive
	\begin{equation}\label{tdg}
	\td g^{ij}
	= \delta^{ij} - \pair{\bd_iV,\bd_jV} + 2\pair{V,A_{ij}} 
	+ Q_{ij}(y,V,D V)
	\end{equation}
	and 
	\begin{equation}\label{tdA}
	    \td A_{ij}
	= A_{ij} + (\bd_{ij}^2 V^\beta) \N_\beta
	- A_{ij}^\beta (\bd_k V^\beta)(\bd_k F)
	+ Q(y,V,D V)
	\end{equation}
	with $|Q_{ij}(y,V,D V)|, |Q(y,V,D V)|\lesssim |y|^{-2} \pr{|V|^2 + |DV|^2}.$ 
    Using \eqref{tdg} and \eqref{tdA}, we can get the expression of the mean curvature $\td H$ of $M,$ and use the same arguments as in \cite{S96} to derive
	$$\M(V_1)-\M(V_2)
	= L(V_1-V_2)
	+ A\cdot (\n^\perp)^2(V_1-V_2)
	+ B\cdot \np(V_1-V_2)
	+ C(V_1-V_2)
	$$
	with 
	$$\sup_\Sigma\pr{|A|+|B|+|C|}
	\lesssim \|V_1\|_{\CS 2} + \|V_2\|_{\CS 2}.$$
	Combining this with Lemma \ref{DV}, the first conclusion follows.
    The second conclusion follows from the same arguments as in \cite[Proof of Lemma 1 in 3.12]{S96} with the estimates above in place of those there.
\end{proof}

At this point, we can follow the same arguments as in \cite{S83} to derive a \L ojasiewicz-Simon inequality for $M.$
For example, we use Theorem \ref{kerL} and Corollary \ref{SolveL} to replace \cite[Equation (2.5)]{S83} and the preceding existence result, and Lemma \ref{Mcon} in place of \cite[Equation (2.10)]{S83}. 
Then, the arguments there will produce the following \L ojasiewicz-Simon inequality for entire vector fields.

\begin{Thm}[Entire \L ojasiewicz-Simon inequality]\label{entireL}
There exists $\beta_0>0,$ $\theta\in(0,1/2)$ and $C>0$ such that if $M$ is the graph over $\Sigma$ of a vector field $U\in\CS{2}$ with $\|U\|_{\CS{2}}\le\beta_0,$ then
\begin{equation*}
|F(M) - F(\Sigma)|^{1-\theta}
\le C\|\M(U)\|_{L^2_W(\Sigma)}
\le C\pr{
\int_M |\varphi|^2 \rho~ dx
}^{1/2}.
\end{equation*}
\end{Thm}

Next, we need to localize the Simon-\L ojasiewicz inequality obtained above, since in a blowup sequence of MCF, we can only write the flow as the graph of a vector field on a large compact set.
To this end, we will introduce some scales related to the concerned self-shrinkers.
For $\theta$ given by Theorem \ref{entireL}, define 
$$
\Theta := 
\left(\frac{1-\frac\theta 2}{1-\theta}\right)^{\frac 1 4} \in \left(1,\left(\frac 32\right)^{\frac 14} \, \right].
$$

\begin{Def}\label{scale}
	For a submanifold $M^n\sbst\bb R^{n+k},$ let $\varphi=\h+\frac{x^\perp}{2},$ and we define the following scales of $M.$
	
	\noindent (1) (Shrinker scale)\footnote{Here we define the shrinker scale following \cite{CS}, and it is different from the shrinker scale defined in \cite{CM15} and \cite{CM19}.} The shrinker scale $\R(M)$ of $M$ is defined by
	$$e^{-\frac{\R(M)^2}{4}}
	= \int_M |\varphi|^2 dx.$$
	
	\noindent (2) (Rough conical scale) For $l\in\bb N$ and $C_l>0$ (to be chosen) fixed, the rough conical scale $\rl(M)$of $M$  is the largest radius $\td r$  such that $M\cap B_{\td r}$ is smooth and
	$$|(\n^\perp)^k A_M(x)|\le \frac{C_l}{(1+r(x))^{1+k}}$$
	for all $k=0,\cdots,l-1$ and $x\in M\cap B_{\td r}.$
	
	\noindent (3) (Conical scale) For a fixed asymptotically conical self-shrinker $\Sigma^n$ and $\beta_0=\beta_0(\Sigma)$ given by the entire \L ojasiewicz inequality (Theorem \ref{entireL}), the conical scale $\rr_l(M)$ of $M$ is the largest radius $r\le\rl(M)$ such that there exists $U\in\CS 2$ with $\|U\|_{\CS2}\le \beta_0$ such that 
	$$\gr U|_{\Sigma\cap B_{r}}\sbst M \text{ and }
	M\cap B_{r-1}\sbst \gr U.$$

	\noindent (4) (Core graphical condition) We say $M$ satisfies the core graphical condition $(*_{b,\udl r})$ if $\rl(M)\ge r$ and there exists a $C^{l+1}$ vector field $U$ over $\Sigma\cap B_{\udl r}$ with $\|U\|_{C^{l+1}}\le b$ such that 
	$$\gr U\sbst M\text{ and }
	M\cap B_{\udl r-1} \sbst \gr U.$$

    \noindent (5) (Roughly conical approximate shrinker) For $R \ge \udl r$, we say that $M^{n}\sbst \bb R^{n+k}$ is a roughly conical approximate shrinker up to scale $R$ if $\Theta R\leq  \tilde{r_\ell}(M),$ $M$ satisfies the core graphical hypothesis $(*_{b, \underline r}),$ and
    $|\phi| + (1+|x|)|\nabla^\perp \phi| \leq s (1+|x|)^{-1}$ on $M\cap B_{\Theta R}$. 
\end{Def}

At this point, we are able to obtain a higher-codimensional version of the graphical estimate as in \cite[Proposition 7.2]{CS}.
The proof is similar, and as in Proposition \ref{Sc}, we need to take the normal projections of the corresponding vector fields.

\begin{Pro}\label{prop:approx}
Taking $\ell$ sufficiently large, there are constants $b,s>0$ sufficiently small, depending on $\Sigma$, $\beta_{0}$, $C_{\ell}$, and $\lambda_0$ such that if $M^{n}\sbst \bb R^{n+k}$ has $\lambda(M) \leq \lambda_0$ and is a roughly conical approximate shrinker up to scale $R$, then there exists $U: \Sigma \to\bb R^k$ with
$$
\gr U|_{\Sigma\cap B_{R}(0)} \sbst M \text{ and }  M\cap B_{R-1}(0)\sbst \gr U
$$
and $\Vert U \Vert_{\CS 2} \leq \beta_{0}$. 
That is to say, the conical scale satisfies $r_{\ell}(M) \geq R$.  
\end{Pro}

Using Proposition \ref{prop:approx}, we can get the final version of the \L ojasiewicz-Simon inequality we would like. 
This can be done based on the same arguments in the proofs of Theorem 6.1 and Theorem 8.1 in \cite{CS}.

\begin{Thm}[Localized \L ojasiewicz-Simon inequality]\label{fLS}
Let $M\sbst \bb R^N$ with $\lambda(M)\le\lambda_0.$
If $\R(M)$ is sufficiently large with $\R(M)\le\rl(M)-1$ and $M$ satisfies the core graphical condition $(*_{b,\udl r}),$ then there exists $C=C(\Sigma, \lambda_0)$ such that
\begin{equation*}
|F(M) - F(\Sigma)|
\le C\pr{
	\int_M |\varphi|^2 \rho~ dx
}^{\frac{1}{2(1-\theta/3 )}}
\end{equation*}
where $\theta$ is the one given by Theorem \ref{entireL}.
\end{Thm}

\section{\bf Uniqueness for asymptotically conical self-shrinkers}\label{sec:4}

Now we are in a position to prove Theorem \ref{prop:uniqueness}.
Fix some sufficiently large $\udl r$ and sufficiently small $\varepsilon$ (which will be chosen in the proof of Proposition \ref{exp}). 
Assume a tangent flow of a MCF $\td M_t$ is modelled on $\Sigma.$
That means that after rescaling, we can obtain a rescaled MCF $M_\tau$ ($\tau\in[-2,\infty)$) 
so that there is a map 
$$U\colon (\Sigma\cap B_{\varepsilon^{-1}}) \times [-2,\varepsilon^{-2}] \to \bb R^k 
$$
such that for all $\tau\in [-2, \varepsilon^{-2}],$ we have
$$\gr U(\cdot,\tau)\sbst M_\tau,
M_\tau\cap B_{\varepsilon^{-1}-1}\sbst \gr U(\cdot,\tau),$$
$\|U\|_{C^{l+1}(\Sigma\cap B_{\varepsilon^{-1}})}\le \varepsilon,$ and
$F(M_\tau)-F(\Sigma)\le\varepsilon.$
As in \cite{CS}, we define the graphical time to be
$$\ovl\tau
:= \sup \{t\in[-2,\infty): M_\tau\text{ satisfies }(*_{b,\udl r})\text{ for all }\tau\in[-2,t]\}.
$$
We are going to establish the following extension result.

\begin{Pro}\label{exp}
	There exists $\udl r_0=\udl r_0(\Sigma,M_{-2})$ sufficiently large such that for $\udl r\ge \udl r_0,$ $\varepsilon=\varepsilon(\Sigma,\udl r)$ sufficiently small and $C_l=C_l(\Sigma,\udl r)$ sufficiently large, we have
	\begin{equation}\label{expimprove}
	\rl\ge \frac 12 e^{\frac \tau 2}\udl r
	\end{equation}
	for all $\tau\in [0,\ovl\tau)$ such that $\R(M_\tau)>\rl(M_\tau).$
\end{Pro}

In the proof of proposition \ref{exp}, the key ingredient is White's Brakke estimate.\footnote{The version we apply here is the one used in \cite{CM15}.}

\begin{Thm}[\cite{W05}]\label{White}
	For $\varepsilon>0$ small enough, there exists a constant $C(\varepsilon)>0$ such that the following holds. If $\hat M_t^k\sbst\bb R^{n+k}$ ($t< 0$) flow by the MCF with $\lambda(\hat M_s)\le\lambda_0,$ and 
	$$(-4\pi t_0)^{-\frac n2} \int_{\hat M_{t_0}} e^{\frac{|x-x_0|^2}{4t_0}} \le \varepsilon
	$$
	for some $x_0\in\bb R^{n+k}$ and $t_0<0,$ then
	$$\sup_{\hat M_t\cap B_{\sqrt{-t_0}/2}(x_0)} |A|^2
	\le \frac{C(\varepsilon)}{-t_0}$$
	for all $t\in[\frac {t_0}4,0).$ Also, the constant $C(\varepsilon)\to 0$ as $\varepsilon\to 0.$	
\end{Thm}

\begin{proof}[Proof of Proposition \ref{exp}]
	Let $\tau_0\in[0,\ovl\tau).$ Consider the flow
	$\hat M_t:=\sqrt{-t} M_{\tau_0-\log(-t)},$
	which is an ordinary MCF (depending on $\tau_0$). Then $\hat M_{-1}=M_{\tau_0}.$ For $\delta_0>0,$ we take $\varepsilon=\varepsilon(\delta_0)>0$ so small that the corresponding constant $C(\varepsilon)$ in theorem \ref{White} satisfies $C(\varepsilon)\le 4 \delta_0.$ Based on \cite[Lemma 2.4]{CIM} and the entropy bound $\lambda_0,$ we can take $R_0=R_0(\delta_0)$ large enough such that
	\begin{equation}\label{entropyout}
	(16\pi)^{-\frac n2} \int_{\hat M_{-4} \setminus B_{R_0}(x_0)} e^{-\frac{|x-x_0|^2}{16}}\le \frac\varepsilon 2
	\end{equation}
	for any $x_0\in\bb R^{n+k}.$
	Next, by taking $\udl r$ large and $b$ small, using Lemma \ref{sff} and $(*_{b,\udl r}),$ we have
	$$|A_{\hat M_{-4}}|^2\le \sigma_0
	\text{ on } \hat M_{-4} \cap (\ovl B_{\udl r}\setminus B_{\udl r/2})$$
	for any given $\sigma_0>0.$
    We may also assume $\udl r\ge 4R_0.$ 
    By definition, if $\sigma_0$ is sufficiently small, we have
    \begin{equation}\label{entropyin}
	\lambda(\hat M_{-4} \cap B_{R_0}(x_0))\le 1 + \frac\varepsilon 2
	\end{equation}
	for $B_{R_0}(x_0)\sbst \ovl B_{\udl r}\setminus B_{\udl r/2}.$
	Combining \eqref{entropyout} and \eqref{entropyin}, we obtain
	\begin{equation}\label{entropy4}
	(16\pi)^{-\frac n2} \int_{\hat M_{-4}} e^{-\frac{|x-x_0|^2}{16}}\le 1+\varepsilon.
	\end{equation}
	Therefore, Theorem \ref{White} implies
	\begin{equation*}
	\sup_{\hat M_t\cap B_{1}(x_0)} |A|^2
	\le \frac{C(\varepsilon)}{4}\le \delta_0
	\end{equation*}
	for $t\in[-1,0]$ and $x_0\in \ovl B_{\udl r-R_0}\setminus B_{\udl r/2+R_0}.$ 
    As a result, we can take a universal constant $\delta_0 = \delta_0(n)$ (such that, for example, $n\delta_0\le\frac 14$) and assume $\hat M_{-1} \cap (\ovl B_{\udl r-R_0-1}\setminus B_{\udl r/2+R_0+1})$ stays in $\ovl B_{\udl r-R_0-1}\setminus B_{\udl r/2+R_0+1}$ under the MCF in the time interval $[-1,0).$ 
	Thus, by applying the Shi-type estimate \cite[Lemma 3]{AB}, we derive
	\begin{equation}\label{|A|bdd}
	|\n^k A_{\hat M_t}|\le C(M_{-2},\lambda_0,w)
	\text{ on }\hat M_{t} \cap \ovl B_{\udl r-R_0-1} 
	\end{equation}
	for $t\in[-1+w,0)$ and $k=0,\cdots, l-1.$
	
	Now we follow the ideas in \cite{CS} to prove the proposition. 
    We may assume \eqref{expimprove} is true for $\tau\in[-2,1]$ by taking $\varepsilon$ small enough. 
    On the other hand, given $\tau\in [1,\ovl \tau)$ and 
	$$x\in M_\tau\cap \left(\ovl B_{e^{\frac \tau 2}(\udl r-R_0-1)} \setminus B_{\udl r} \right),$$
	consider
	$$\tau_0:= \tau + 2\log\frac{\udl r-R_0-1}{|x|}\in [0,\tau).$$
	Then,
	$$t := -e^{\tau_0-\tau} = -\frac{(\udl r-R_0-1)^2}{|x|^2} \in [-1+w,0)$$
	where
	$w:=1-\left(1-\frac{R_0-1}{\udl r}\right)^2\in (0,1).
	$
	After rescaling, we get
	$$\hat x:=\sqrt{-t}x\in \hat M_{t}\cap \bd B_{\udl r-R_0-1}.$$
	Consequently, the curvature estimate \eqref{|A|bdd} implies
	$$|\hat x|^{-1-k} |\hat x|^{1+k} |\n^k A_{\hat M_t}(x)| 
	= |\n^k A_{\hat M_t}(x)| \le C=C(M_{-2},\lambda_0,w).$$
	Rescaling back, we have
	$$|x|^{1+k}|\n^k A_{M_\tau}(x)|\le C(\udl r-R_0+1)^{1+k}.$$
	Finally, choosing $C_l=(M_{-2},\lambda_0,w)$ sufficiently large in (2) of definition \ref{scale}, we have
	$$\rl(M_\tau) 
	\ge e^{\frac \tau 2}(\udl r-R_0-1)\ge \frac 12 e^{\frac \tau 2}
	$$
	by our choices of $R_0.$	
\end{proof}

With the \L ojasiewicz-Simon inequality (Theorem \ref{fLS}) and the extension result (Proposition \ref{exp}), we are able to prove the higher-codimensional version of the theorem proven by Chodosh-Schulze.

\begin{proof}[Proof of Theorem \ref{prop:uniqueness}]
    After establishing Theorem \ref{fLS} and Proposition \ref{exp}, we can follow the arguments in \cite[Section 9]{CS}, which are standard after having a \L ojasiewicz-type inequality, verbatim to get the uniqueness of asymptotically conical tangent flows.
\end{proof}

\end{document}